\documentclass{amsart}
\usepackage[T1]{fontenc}
\usepackage[english]{babel}
\usepackage{amsmath,amsthm,amssymb, latexsym,geometry,fancyhdr,stackrel,enumerate,bbm}
\usepackage[all]{xy}
\usepackage{srcltx}
\usepackage[left,modulo, pagewise]{lineno}
\usepackage{hyperref}

\date{\today}
\swapnumbers
\theoremstyle{plain}
\newtheorem{thm}{Theorem}[section]
\newtheorem{lem}[thm]{Lemma}
\newtheorem{prop}[thm]{Proposition}
\newtheorem{cor}[thm]{Corollary}

\theoremstyle{definition}
\newtheorem{de}[thm]{Definition}
\newtheorem{rem}[thm]{Remark}
\newtheorem{rems}[thm]{Remarks}
\newtheorem{ex}[thm]{Example}
\newtheorem{algo}[thm]{Algorithm}

\newcommand{\ts}[1]{\normalfont{\textsf{#1}}}
\newcommand{\K}{\ts k}

\renewcommand{\a}{\alpha}
\renewcommand{\b}{\beta}
\renewcommand{\c}{\gamma}
\newcommand{\e}{\varepsilon}

\newcommand{\s}{\sigma}

\newcommand{\G}{\Gamma}

\newcommand{\mor}[3]{$#1\colon #2 \to #3$}

\newcommand{\rad}[1]{\ts{rad} #1}
\renewcommand{\dim}[1]{\ts{dim}_{\ts k}\  #1 }

\newcommand{\Hom}[3]{\ts{Hom}_{#1}(#2,\,#3)}
\newcommand{\Ext}[4]{\ts{Ext}_{#1}^{#2}(#3,\,#4)}
\newcommand{\End}[2]{\ts{End}_{#1}(#2)}

\newcommand{\HH}[2]{\ts{HH}^{#1}(#2)}

 \title[Derived class of $m$-cluster tilted algebras of type $\mathbb{A}$]{Hochschild cohomology and the derived class of $m$-cluster tilted algebras of type $\mathbb{A}$}

\author[J.C. Bustamante]{Juan Carlos Bustamante}
\address{D\'epartement de Math\'ematiques, Universit\'e
  de Sherbrooke, Sherbrooke, Qu\'ebec, Canada J1K~2R1}
\email{juan.carlos.bustamante@usherbrooke.ca}

\author[V. Gubitosi]{Viviana Gubitosi}
\address{D\'epartement de Math\'ematiques, Universit\'e
  de Sherbrooke, Sherbrooke, Qu\'ebec, Canada J1K~2R1}
\email{viviana.gubitosi@usherbrooke.ca}

\keywords{$m$-cluster tilted algebras; Hochschild cohomology; gentle algebras}
\begin{document}
\maketitle
\section*{Introduction} Cluster categories were introduced in \cite{BMRRT06} in order to model the combinatorics of the cluster algebras of Fomin and Zelevinski \cite{FZ02}, using  tilting theory of hereditary finite dimensional algebras over algebraically closed fields. The clusters correspond to the tilting objects in the cluster category.

Roughly speaking, given an hereditary finite dimensional algebra $H$ as above, the cluster category $\mathcal{C}_H$ is obtained from the derived category $\mathcal{D}^b(H)$ by identifying the shift functor $[1]$ with the Auslander - Reiten translation $\tau$. By a result of Keller \cite{K05}, the cluster category is triangulated, and the same holds for  the category obtained from $\mathcal{D}^b(H)$ by identifying the composition $[m]:=[1]^m$ with $\tau$. The latter is called an $m$-cluster category, in which $m$-cluster tilting objects have been defined by Thomas, in \cite{T07}, who in addition showed that they are in bijective correspondence with the $m$-clusters of Fomin and Reading \cite{FR05}. The endomorphisms algebras of the $m$-cluster tilting objects are called $m$-cluster tilted algebras or, in case $m=1$, cluster tilted algebras.

Caldero, Chapoton and Schiffler gave in \cite{CSS06} an interpretation of the cluster categories $\mathcal{C}_H$ in case $H$ is hereditary of Dynkin type $\mathbb{A}$ in terms of triangulations of the disc with marked points on its boundary, using an approach also present in \cite{FST08} in a much more general setting. This has been generalized by several authors. For instance, Baur and Marsh in \cite{BM08}, considered $m$-angulations of the disc, modelling the $m$-cluster categories. In \cite{ABCP09}, Assem \emph{et al.} showed that cluster tilted algebras coming from triangulations of the disc or the annulus with marked points on their boundaries are gentle, and, in fact, that these are the only gentle cluster tilted algebras. The class of gentle algebras has been extensively studied in \cite{AH81, AG08, BB10, Buan-Vatne08, M10, SZ03}, for instance, and is particularly well understood, at least from the representation theoretic point of view. This class includes, for instance, iterated tilted, and cluster tilted algebras of types $\mathbb{A}$ and $\tilde{\mathbb{A}}$, and, as shown in \cite{SZ03}, is closed under derived equivalence. 

When studying module categories, one is often interested in them up to derived equivalence, or tilting-cotilting equivalence. In \cite{Buan-Vatne08}, Buan and Vatne gave a criterion to decide whether two cluster tilted algebras of type $\mathbb{A}$ are themselves derived equivalent or not. This has been done using the Cartan matrix as derived invariant, as well as mutations of quivers. Later, Bastian, in \cite{Bastian09}, gave an analogous classification for the $\tilde{\mathbb{A}}$ case. She used another thinner derived invariant, the function $\phi$ introduced by Avella-Alaminos and Geiss in \cite{AG08}. In \cite{BB10} a more general question has been considered, namely the characterization of the algebras that are derived equivalent to cluster tilted algebras of type $\mathbb{A}$ or $\tilde{\mathbb{A}}$. Again, in this paper the map $\phi$ is of central importance, and the characterizations therein are given in terms of the form of this map. In  another direction, results analogous to those of \cite{Buan-Vatne08} have been established for $m$-cluster tilted algebras of type $\mathbb{A}$ by Murphy in \cite{M10}: he described these algebras by quivers and relations, and gave a criterion permitting to decide whether two $m$-cluster tilted algebras of type $\mathbb{A}$ are derived equivalent or not. Again, he used the Cartan matrix as in \cite{Buan-Vatne08}, but ``elementary polygonal moves'' instead of  -- but equivalently to -- mutations.

The aim of this paper is to classify the algebras that are derived equivalent to $m$-cluster tilted algebras of type $\mathbb{A}$. Our approach differs from those of \cite{Buan-Vatne08, BB10, Bastian09, M10} in the fact that we use the Hochschild cohomology ring as derived invariant. Since the algebras we are interested in are gentle, the required computations can be done efficiently. It is a thinner invariant than the Cartan matrix, as it detects not only the number of cycles with full relations, but also their length, and it distinguishes the case where the characteristic of the field is $2$. In our case,  the structure of $\HH{\ast}{A}$ is much easier to compute than the map $\phi_A$. 

Recall that given a connected quiver $Q$, its \emph{Euler characteristi}c, $\chi(Q)=|Q_1|-|Q_0|+1$, is the rank of its first homology group, $Q$ being viewed as a graph.

The main result of this paper can be stated as follows:

\subsection*{Theorem A}\textit{ A connected algebra $A=\K Q/I$ is derived equivalent to a connected component of an $m$-cluster tilted algebra of type $\mathbb{A}$ if and only if  $(Q,I)$ is a gentle bound quiver having $\chi(Q)$ oriented cycles of length $m+2$, each of which has full relations. }

\medskip

In particular, specializing to the case $m=1$, recover known results of \cite{Buan-Vatne08, BB10}, and we obtain a criterion to decide whether or not an algebra is derived equivalent to a cluster tilted algebra of type $\mathbb{A}$. The latter is very easy to use, as it does not require any computation, in contrast with the known result of \cite{BB10}. Although it is not known if the function $\phi$ is a complete invariant in general, in \cite{BB10} it is shown that this is indeed the case for algebras derived equivalent to cluster tilted algebras of type $\mathbb{A}$. We generalize this result to the $m$-cluster tilted context.  Moreover, from our characterizations, we deduce the \emph{invariant pair} $(r,s)$ of an $m$-branched algebra $A=\K Q / I$: The integer $s$ is the rank of the Grothendieck group $K_0(A)$, while $r$ is the Euler characteristic of $Q$, which in addition parametrizes the structure of the  Hochschild cohomology ring  $\HH{\ast}{A}$. In case $A$ is connected, it is also the rank of the fundamental group $\pi_1(Q,I)$ of the bound quiver $(Q,I)$, which is free in this context.  Summarizing, we have:

\subsection*{Theorem B}\textit{ Let $A=\K Q / I$ and $A'=\K Q'/I'$ be connected $m$-branched algebras. Then the following conditions are equivalent.
\begin{enumerate}
 \item[$a)$] $A$ and $A'$ are derived equivalent,
\item[$b)$] $A$ and $A'$ are tilting-cotilting equivalent,
\item[$c)$] $A$ and $A'$ have the same invariant pair,
\item[$d)$] $\HH{\ast}{A} \simeq \HH{\ast}{A'}$ and $K_0(A) \simeq K_0(A')$,
\item[$e)$] $\phi_A = \phi_{A'}$.
\item[$f)$] $\pi_1(Q,I) \simeq \pi_1(Q',I')$ and $|Q_0|  = |Q'_0|$.
\end{enumerate}}

\medskip
The paper is organized as follows: In section 1 we recall facts about gentle algebras, derived and tilting-cotiling equivalences, and $m$-cluster tilted algebras. In section 2 we establish the facts about Hochschild cohomology that will be used in the sequel. In section 3 we introduce what we call $m$-branched algebras, precisely those described in the theorem above, and in section 4 we recall what Brenner-Butler tilting modules are. They are used in section 5 to remove the relations in the quiver of a branched algebra that do not lie in a cycle. Section 6 is devoted to the proof of the main theorem and some consequences, among which, in section 7, we recover the known results mentioned above.

\section{Preliminaries}
\subsection{Gentle algebras}
While we briefly recall some  concepts concerning bound quivers and algebras, we refer the reader to \cite{ASS06} or \cite{ARS95}, for instance, for unexplained notions.

Let \K \  be a commutative field. A quiver $Q$ is the data of two sets, $Q_0$ (the \textit{vertices}) and $Q_1$ (the \textit{arrows}) and two maps \mor{s,t}{Q_1}{Q_0} that assign to each arrow $\a$ its \textit{source} $s(\a)$ and its \textit{target} $t(\a)$. We write \mor{\a}{s(\a)}{t(\a)}. If $\b\in Q_1$ is such that $t(\a)=s(\b)$ then the composition of $\a$ and $\b$ is the path $\a\b$. This extends naturally to paths of arbitrary positive length. The \emph{path algebra} $\K Q$ is the $\K$-algebra whose basis is the set of all paths in $Q$, including one stationary path $e_x$ at each vertex $x\in Q_0$, endowed with the  multiplication induced from the composition of paths. In case $|Q_0|$ is finite, the sum of the stationary paths  - one for each vertex - is the identity.
 
If the quiver $Q$ has no oriented cycles, it is called \emph{acyclic}. A \emph{relation} in $Q$ is a $\K$-linear combination of paths of length at least $2$ sharing source and target.  A relation which is a path is called \emph{monomial}, and the relation is \emph{quadratic} if the paths appearing in it have all length $2$. Let $\mathcal{R}$ be a set of relations. Given an integer $n$, an \emph{$n$-cycle} is an oriented cycle consisting of $n$ arrows, and it is called \emph{an $n$-cycle with full relations} if the composition of any two consecutive arrows on this cycle belongs to $\mathcal{R}$. In the sequel, by cycle we mean oriented cycle. Given $\mathcal{R}$ one can consider the two-sided ideal of $\K Q$ it generates $I=\langle \mathcal{R}\rangle \subseteq  \langle Q_1 \rangle^2$. It is called \emph{admissible} if there exists a natural number $r\geqslant 2$ such that $\left\langle Q_1 \right\rangle^r \subseteq I$. The pair $(Q,I)$ is a \emph{bound quiver}, and associated to it is the algebra $A=\K Q/I$. It is known that any finite dimensional basic algebra over an algebraically closed field is obtained in this way, see \cite{ASS06}, for instance. 

Recall from \cite{AH81} that an algebra  $A= \K Q/I$ is said to be \emph{gentle} if  $I=\langle \mathcal{R} \rangle $, with $\mathcal{R}$ a set of monomial quadratic relations such that :
\begin{enumerate}
 \item[G1.] For every vertex $x\in Q_0$ the sets $s^{-1}(x)$ and $t^{-1}(x)$ have cardinality at most two;
 \item[G2.] For every arrow $\a\in Q_1$ there exists at most one arrow $\b$ and one arrow $\c$ in $ Q_1$ such that $\a\b\not\in I$, $\c\a\not\in I $;
 \item[G3.] For every arrow $\a\in Q_1$ there exists at most one arrow $\b$ and one arrow $\c$ in $Q_1$ such that $\a\b\in I$, $\c\a\in I$.
\end{enumerate}

Gentle algebras are special biserial (see \cite{WW85}), and have extensively been studied in several contexts, see for instance \cite{AG08, BB10, Buan-Vatne08, Bus06-cohomo, M10, SZ03}.

\subsection{Tilting-cotilting, and derived equivalences}
Given a finite dimensional algebra $A=\K Q/I$ a \emph{tilting} module is a finitely generated right $A$-module of projective dimension less than or equal to 1, having no self extensions and exactly $|Q_0|$ indecomposable non isomorphic direct summands, \cite{ASS06}. The notion of cotilting module is defined dually. Given a tilting $A$-module $T$, with $A$ hereditary, the algebra $\End{A}{T}$ is said to be \emph{tilted}. Two algebras $A$ and $B$ are said to be \emph{tilting-cotilting} equivalent it there exists a finite sequence of algebras $A=A_0, A_1,\ldots, A_r=B$ and $A_i$-tilting (or cotilting) modules $T_i$ such that $A_{i+1}=\End{A_{i}}{T_i}$ for $i\in\{0,1,\ldots, r-1\}$.

Denote by $\mathcal{D}^b(A)$ the bounded derived category of finite dimensional right $A$-modules. Its objects are bounded complexes of finite dimensional right $A$-modules, and morphisms are obtained from morphisms of complexes by localizing with respect to quasi-isomorphisms (see \cite{H88}). The category $\mathcal{D}^b(A)$ is triangulated, with translation functor induced by the shift of complexes. Two algebras $A$ and $B$ are \emph{derived equivalent} if the categories $\mathcal{D}^b(A)$ and $\mathcal{D}^b(B)$ are equivalent as triangulated categories. It has been shown by Happel \cite{H88} that if two algebras are tilting-cotilting equivalent, then they are derived equivalent. Moreover, Schr\"oer and Zimmermann showed in \cite{SZ03} that the class of gentle algebras is stable under derived equivalence.

\subsection{$m$-cluster tilted algebras} Let $H\simeq \K Q$ be an hereditary algebra such that $Q$ is acyclic. The derived category $\mathcal{D}^b(H)$ is triangulated, the translation functor, denoted by $[1]$, being induced from the shift of complexes. For an integer $n$, we denote by $[n]$ the composition of $[1]$ with itself $n$ times, thus $[1]^n = [n]$. In addition, $\mathcal{D}^b(A)$ has Auslander-Reiten triangles, and, as usual, the Auslander-Reiten translation is denoted by $\tau$. 

Let $m$ be a natural number. The $m$-cluster category of $H$ is the quotient category $\mathcal{C}_m(H):=\mathcal{D}^b(H)/ \tau^{-1} [m]$ which carries a natural  triangulated structure, see \cite{K05}. Following \cite{T07} we consider \textit{$m$-cluster tilting objects}  in $\mathcal{C}_m(H)$ defined as objects satisfying the following conditions:
\begin{enumerate}
 \item $\Hom{\mathcal{C}_m(H)}{T}{X[i]} = 0$ for all $i\in\{1,2,\ldots,m\}$ if and only if $X\in \ts{add}\ T$,
 \item $\Hom{\mathcal{C}_m(H)}{X}{T[i]} = 0$ for all $i\in\{1,2,\ldots,m\}$ if and only if $X\in \ts{add}\ T$.
\end{enumerate}
The endomorphism algebras of such objects are called \emph{$m$-cluster tilted algebras of type $Q$}. In case $m=1$, this definition specializes to that of a cluster tilted algebra, a class intensively studied since its definition in \cite{BMR06}.  

In \cite{ABCP09} it has been shown that cluster tilted algebras are gentle if and only if they are of type  $\mathbb{A}$ or $\tilde{\mathbb{A}}$. In \cite{Buan-Vatne08} Buan and Vatne gave the derived equivalence classification of cluster tilted algebras of type $\mathbb{A}$. They showed that two cluster tilted algebras of type $\mathbb{A}$ are derived equivalent if and only if their quivers have the same  number of 3-cycles with full relations.  Later, in \cite{Bastian09} the same work has been done for cluster tilted algebras of type $\tilde{\mathbb{A}}$. Moreover, in \cite{BB10}, the algebras that are derived equivalent to cluster tilted algebras of types $\mathbb{A}$ or $\tilde{\mathbb{A}}$ have been classified. In this classification, again, combinatorial data of the involved bound quiver is of central importance.

On the other hand, using arguments similar to those of \cite{ABCP09}, Murphy showed in \cite{M10} that $m$-cluster tilted algebras are gentle and  that the only cycles that may exist in their quivers are $(m+2)$-cycles, each of which has full relations. Perhaps the most noticeable differences between cluster tilted and $m$-cluster tilted algebras is that the latter need not to be connected, and that the quiver does not determine the algebras. Indeed, in case $m>1$, there can be relations outside the cycles, but there are at most $m-1$ consecutive such relations.  Also, he described the connected components of $m$-cluster tilted algebras up to derived equivalence, a result analogous to that of \cite{Buan-Vatne08}, namely: two connected components of $m$-cluster tilted algebras of type $\mathbb{A}$ are derived equivalent if and only if their quivers have the same number of $(m+2)$-cycles (theorem 1.2 in \cite{M10}). Moreover, he showed that any connected component of an $m$-cluster tilted algebra can be reduced, using local mutations (which are tiltings or cotiltings, as we shall see) to what he calls the \emph{normal form}, described as follows. Recall that $m$ is fixed. Let $Q_{r,s}$ be the quiver having $r$ cycles of length $m+2$ and $s$ vertices, depicted below: 

$$\SelectTips{eu}{10}
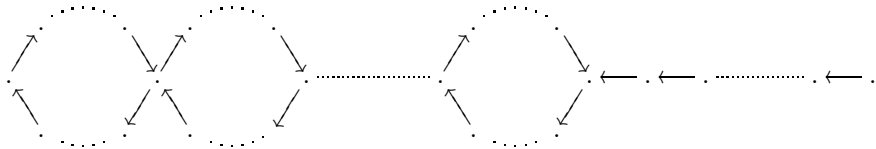
\xymatrix@C=.3pc@R=.3pc{  & .\ar@{.}@/^/[rrr] &&& .\ar[ddr] &  & . \ar@{.}@/^/[rrr] &&& .\ar[ddr] &  &  &&&&  & . \ar@{.}@/^/[rrr] &&& . \ar[ddr] &  &&&&&&&&&&\\
&&&&&&&&&&&&&&&&&&&&  &&&&&&&&&&\\
. \ar[ruu] & &&& & . \ar[ddl] \ar[uur] &  &&&  & . \ar[ddl] \ar@{.}[rrrrr] & &&&&  .\ar[uur] & &&& & .\ar[ddl]  && .\ar[ll]  && . \ar@{.}[rrrr] \ar[ll] &&&& . && .\ar[ll]\\
&&&&&&&&&&&&&&&&&&&&  &&&&&&&&&& \\
& .\ar@{.}@/_/[rrr]  \ar[uul] &&& . &  & . \ar@{.}@/_/[rrr] \ar[uul] &&&  &  &  &&&&  & . \ar@{.}@/_/[rrr] \ar[uul] &&& .  &    &&&&&&&&&&
}$$

\noindent Let $I$ be the ideal generated by the composition of any two consecutive arrows in any of the cycles of $Q_{r,s}$, and $N_{r,s} = \K Q_{r,s}/I$ be the corresponding algebra, which is an $m$-cluster tilted algebra of type $\mathbb{A}$. Note that in \cite{M10} the definition of the normal form is slightly different. On one hand the arrows of the linear part (those at the right of the diagram) are oriented in the other way. On the other hand, the number of arrows in each cycle that lie between the vertices of degree greater than 2 is constant. This more restrictive definition is not necessary four our purpose, since this gives an algebra which is clearly tilting cotilting equivalent to ours.

Moreover, again from \cite{M10}, we know that a connected component of an algebra derived equivalent to an $m$-cluster tilted algebra is derived equivalent to a normal form  $N_{r,s}$. In case $m=1$, this form coincides with the one considered by Buan and Vatne \cite[2.3]{BB10}.

\medskip In the sequel, each oriented cycle depicted should be understood as an $(m+2)$-cycle  with full relations.

\section{Hochschild cohomology}

Given an algebra $A$, the $n$-th Hochschild cohomology group of $A$ with coefficients in the bimodule $_A A_{A}$ is the extension group $\ts{HH}^n(A)=\Ext{A-A}{^n}{A}{A}$. The sum $\ts{HH}^\ast (A) = \bigoplus_{n\geqslant 0} \ts{HH}^n(A)$ has the additional structure of a Gerstenhaber algebra, see \cite{G63}. From \cite{R89, K04} this structure is known to be a derived invariant, that is, invariant under derived equivalence. We shall use this invariant instead of the function $\phi$ of Avella-Alaminos and Geiss \cite{AG08}, invariant used in \cite{BB10} and \cite{Bastian09}.

We are interested in the computation of the cohomology groups of  $N_{r,s}$. Since the latter are schurian, one can prove, using \cite[2.2]{Hap89} for instance, that $\ts{dim}_{\ts{k}} \HH{1}{N_{r,s}} = r$. However the first cohomology group misses much information, it  takes into account the number of cycles, but does not retain their length, nor  whether they are bound with relations or not. In what follows we see that the higher cohomology groups do detect what is missed by $\HH{1}{N_{r,s}}$.

In order to carry out  the needed computations, we introduce a convenient complex that allows us to do so efficiently, see \cite{B97, Bus06-cohomo}. Let $A=\K Q /I$ be a monomial quadratic algebra. Define $\G_0=\G_0(Q)=Q_0,\ \G_1=\G_1(Q)=Q_1$, and for $n\geqslant 2$, $\G_n=\G_n(Q,I)=\{\a_1\a_2\cdots \a_n|\ \a_i\a_{i+1}\in I\}$. Moreover, let $E=\ts{k}Q_0$ be the semi-simple algebra isomorphic to $A/\rad{A}$, and $\ts{k}\G_n$ the $\ts{k}$-vector space with basis $\G_n$. The latter are also $E-E$-bimodules in an obvious way. In what follows, tensor products are taken over $E$. With these notations we have a minimal projective resolution of $_A A_A$:
$$\SelectTips{eu}{10}\xymatrix@R=5pt{\cdots \ar[r] & A\otimes \K \G_n \otimes  A\ar[r]^-{\delta_{n-1}}& A\otimes \K \G_{n-1} \otimes  A \ar[r]^-{\delta_{n-2}}&\cdots &\\
\cdots \ar[r] & A\otimes \K \G_1 \otimes  A\ar[r]^-{\delta_{0}}& A\otimes \K \G_{0} \otimes  A \ar[r]^-{\epsilon}&A \ar[r] &0}$$
where $\epsilon$ is the composition of the isomorphism $A\otimes \K \G_0 \otimes A \simeq A\otimes A$ with the multiplication of $A$, and given  the differential $\delta_{n-1}$  defined on $A\otimes \K \G_n \otimes A$ by the formula $$\delta_{n-1}(1\otimes \a_1\cdots \a_n\otimes 1) = \a_1 \otimes \a_2 \cdots \a_n\otimes 1 + (-1)^{n} 1\otimes \a_1 \cdots \a_{n-1} \otimes \a_n.$$

\begin{lem}
Let $A=\ts{k}Q/I$ be a monomial quadratic algebra. Then the Hochschild cohomology groups $\HH{n}{A}$ are the cohomology groups of the complex
$$\SelectTips{eu}{10}\xymatrix{\cdots \ar[r] & \Hom{E-E}{\ts{k}\G_n}{A}\ar[r]^{\delta^n}& \Hom{E-E}{\ts{k}\G_{n+1}}{A}\ar[r]&\cdots  }$$
where $\delta^n(f)(\a_1\a_2\cdots \a_{n+1}) = \a_1f(\a_2\cdots\a_{n+1})+(-1)^{n+1}f(\a_1 \cdots a_n)\a_{n+1} $.
\end{lem}
\begin{proof}In order to compute the Hochschild cohomology groups of the algebra $A$, we apply the functor $\Hom{A-A}{-}{A}$ to the complex obtained by truncating the resolution above. In addition we use the identification $\Hom{A-A}{A\otimes \K \G_n \otimes A}{A} \simeq \Hom{E-E}{\K \G_n}{A}$.
\end{proof}

As mentioned before, the sum $\ts{HH}^\ast (A) = \bigoplus_{n\geqslant 0} \ts{HH}^n(A)$ has additional structure given by two products, which we now describe, see \cite{G63}. The two products are defined using the standard resolution of $A$, but using appropriate explicit maps between the resolutions, see \cite[Section 2]{SF08} and \cite[Section 1]{Bus06-cohomo}, we can carry them to our context.

Given $f\in \Hom{E-E}{\K\G_n}{A}$, $g\in\Hom{E-E}{\K\G_m}{A}$, and $i\in\{1,2,\ldots n\}$, define the element $f\circ_i g$ as $f \left(\mathbbm{l}^{i-1} \otimes g \otimes \mathbbm{l}^{n-i}\right)$. In addition, define the \textit{composition product} as
$$f \circ g = \sum_{i=1}^n (-1)^{(i-1)(m-1)} f\circ_i g$$
and the \emph{bracket} to be
$$[f,g]= f\circ g - (-1)^{(n-1)(m-1)}g \circ f$$

On the other hand, denote by \mor{\sigma}{A\otimes A}{A} the multiplication of $A$. The \emph{cup-product} $f\cup g$ of $f$ and $g$ is  the  element of $\Hom{E-E}{\K \G_{n+m}}{A}$ defined by $f\cup g = \sigma(f \otimes g)$.

Recall that a Gerstenhaber algebra is a graded $\K$-vector space $A$ endowed with a product which makes $A$ into a graded commutative algebra, and a bracket $[ - ,  - ]$ of degree $−1$ that makes $A$ into a graded Lie algebra, and such that $[x, yz] = [x, y]z + (-1)^{(|x|-1)|y|} y[x, z]$, that is, a graded analogous of a Poisson algebra. The cup product $\cup$ and the bracket $[ - , - ]$ defined above define products in $\ts{HH}^\ast (A)$, which becomes then a Gerstenhaber algebra (see \cite{G63}).

\begin{ex}[Cibils, \cite{C98}]\label{ex:cycle} With the notations of the previous section, let $A=N_{1,m+2}$, that is the Nakayama algebra given by an oriented cycle of length $m+2$, whose radical is $2$-nilpotent. In \cite[4.6]{C98} Cibils gave an explicit formula for the structure of Hochschild cohomology of $A$: if the characteristic of the ground field $\ts{k}$ is  $2$, or if $m$ is even, then put $d=2(m+2)$, otherwise put $d=m+2$. The space $\K \G_d$ has dimension $d$. Define  \mor{f}{\K\G_{d}}{A} as the map that sends each basis element of $\K\G_d$ to its source-target vertex, and let $F$ be its (non zero) class in $\HH{d}{A}$. In addition, define  \mor{g}{\K  \G_{1}}{A} to be  the map that sends a particular arrow  $\a$ on the cycle  to itself  and vanishes on the other basis elements of $\K\G_{1}$, and let $G$ be its (non zero) class in $\HH{1}{A}$. Considering $F$ and $G$ as indeterminates of degrees $d$ and $1$, $\HH{\ast}{A}$ has the multiplicative structure of  $\ts{k}[F,G]/\langle G^2\rangle$. Moreover, a straightforward computation gives  that in addition the bracket is given by
\begin{center}
$\begin{array}{lllll}
[F,G] = F, &  & [F,F]=0,  &  &[G,G]=0. 
\end{array}$
\end{center}

\end{ex}

We now compute the Hochschild cohomology of the normal form $N_{r,s}$. 

\begin{prop} \label{prop:HH} \ 

\begin{enumerate}
 \item[$a)$] For any $n\in \mathbb{N}$, the Hochschild cohomology group  $\HH{n}{N_{r,s}}$ is given by 
$$\HH{n}{N_{r,s}}=\left\{\begin{array}{ll} 
                    \K 					& \mbox{ if } n=0,\\
		\prod_{j=1}^r \HH{n}{N_{1,m+2}} 	& \mbox{ otherwise.} \end{array}\right. $$
 \item[$b)$] With the notations of  example \ref{ex:cycle},  the $r$ elements of degrees $1$ and $d$ that generate $\HH{\ast}{N_{r,s}}$ as a ring are given by  $\{G_j,\ F_j|\ 1\leqslant j \leqslant r\} $ in bijective correspondence with each one of the cycles in the quiver, as in \ref{ex:cycle}.
 \item[$c)$] Each of the $F_i$'s is orthogonal to each other, and to each $G_i$, and the Lie structure is given by: 
\begin{center}
$\begin{array}{lllll}
[F_i, G_j] = \delta_{i,j}F_i, 	&  	& [F_i,F_j]=0,  				&  	& [G_i,G_j]=0 .
\end{array}$ 
\end{center}

\end{enumerate}
 \end{prop}

\begin{proof}
 First of all, note that using  Happel's long exact sequence \cite[5.3]{Hap89} one sees that the Hochschild cohomology of $N_{r,s}$ is precisely  that of the case in which every vertex of $Q$ belongs to a cycle, say $N_{r}$ for short. 

We now exhibit a short exact sequence of complexes from which the result will follow.

To begin, label the cycles of $Q$ by $1,2,\ldots,r$, from right to left. For $j\in\{1,\ldots,r\}$, let $Q^j$ be the full subquiver of $Q$ generated by the arrows of the $j^{th}$ cycle. In addition, let $I^j$ be the restriction of $I$ to $\K Q^j$, and $A^j=\K Q^j / I^j$ the corresponding algebra, which is isomorphic to $N_1 = N_{1,m+2}$, and $E^j =\ K Q^j_0$, its semi-simple part. In addition, for $j<r$,  let $v_j$ be the vertex belonging to $Q^j$ and $Q^{j+1}$.  Accordingly, for $n\geqslant 0$, define $\G_n^j = \G_n(Q^j,I^j)$. It follows directly from the definition  that $\G_n = \coprod_{j=1}^r \G_n^j$ for $n\geqslant 1$.

Now, for $n\geqslant 0$, let \mor{\phi_n}{\underset{1\leqslant j \leqslant r}{\bigoplus} \K \G_n^j}{\K \G_n} be the sum of the morphisms induced by the inclusions of the $\G_n^j$ into $\G_n$, so that $\phi_\bullet$ is a morphism of complexes. For $n\geqslant 1,\ \phi_n$ is an isomorphism, and the kernel $K_0$ of  $\phi_0$  is the $(r-1)$-dimensional vector space $\K v_1\oplus \cdots \oplus \K v_{r-1} $. We thus have a short exact sequence of complexes
$$\SelectTips{eu}{10}\xymatrix{0 \ar[r] & K_\bullet \ar[r]^-{\iota_\bullet}&\displaystyle \bigoplus_{j=1}^r \K \G_\bullet^j\ar[r]^{\phi_\bullet}&\K \G_\bullet\ar[r]&0 }$$
with $K_\bullet$ concentrated in degree $0$. Upon applying the functor $\Hom{E-E}{-}{A}$, and using the natural isomorphism 
$$\Hom{E-E}{\K\G_n}{A} \simeq \bigoplus_{j=1}^r \Hom{E^j-E^j}{\K\G_n^j}{A^j}.$$
we obtain a long exact sequence of cohomology groups
$$ \SelectTips{eu}{10}\xymatrix@C=10pt@R=1pt{0 \ar[r]&\HH{0}{N_{r,s}}\ar[r] &\displaystyle \bigoplus_{j=1}^r \HH{0}{A^j}\ar[r]&\displaystyle \bigoplus_{j=1}^{r-1} \HH{0}{\K v_j}\ar[r]& \HH{1}{N_{r,s}}\ar[r] &\displaystyle \bigoplus_{j=1}^r \HH{1}{A^j}\ar[r]&0\ar[r]&\cdots\\
&& \cdots0\ar[r]& \HH{i}{N_{r,s}} \ar[r] & \displaystyle \bigoplus_{j=1}^r \HH{i}{A^j}\ar[r]&0\cdots& &    }$$
Since the dimension of $\HH{1}{N_{r,s}}$ equals that of $\underset{1\leqslant j \leqslant r}{\bigoplus} \HH{1}{A^j}$, we are done.
\end{proof}

\begin{rems}\label{rems:core}\ 

\begin{enumerate}
 \item[$a)$] As in example \ref{ex:cycle}, the non-nilpotent elements of degree $d$ in the Hochschild cohomology ring correspond to the cycles of length $m+2$ with full relations. This is also true in a more general context, that of monomial quadratic algebras, a class which includes gentle algebras.
 \item[$b)$] In the same way,  the elements of degree $1$ in the Hochschild cohomology ring of $N_{r,s}$  correspond to the choice of an arrow in each such cycle.  Again, this holds in the more general setting of monomial algebras. Indeed any choice of an arrow $\a$ in each cycle $\c$ --  oriented or not -- of $\overline{Q}$ (the non-oriented graph associated to $Q$) gives rise to an element in $\HH{1}{A}$. More precisely the cycle $\c$ gives an element $\tilde{\c}$ in the fundamental group $\pi_1(Q,I)$ of the bound quiver $(Q,I)$, $\pi_1(Q,I)$ (see \cite{AdlP96}) which induces a character $\tilde{\c}^\ast$ in $\Hom{}{\pi_1(Q,I)}{\K^+}$. The required derivation is then the image of  $\tilde{\c}^\ast$ by the monomorphism \mor{s}{\Hom{}{\pi_1(Q,I)}{\K^+}}{\HH{1}{A}} of \cite{AdlP96} (see also \cite{PS01}).

\end{enumerate}
\end{rems}

\section{Branched algebras}
As in the case $m=1$ the class of $m$-cluster tilted algebras of type $\mathbb{A}$ is not closed under derived equivalence; that is,  it is possible for an $m$-cluster tilted algebra  $(m\geqslant 2)$ to be derived equivalent to an algebra which is not $m$-cluster tilted. 

\begin{ex}[Remark 4.9, \cite{M10}]\label{ex:branched}
Consider the following quiver $Q$:
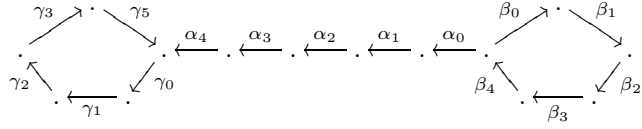
$$\SelectTips{eu}{10}\xymatrix@C=.4pc@R=.2pc{ & & . \ar[ddrr]^{\c_5} & & &&&&&&&&&&& &. \ar[ddrr]^{\b_1}&&\\
&  & & & &&&&&&&&&&&&&&\\
. \ar[uurr]^{\c_3}& & & & .\ar[ddl]^{\c_0} & & . \ar[ll]_{\a_4} & & .\ar[ll]_{\a_3} & & .\ar[ll]_{\a_2} & & .\ar[ll]_{\a_1}  & & . \ar[ll]_{\a_0}  \ar[uurr]^{\b_0} & &&& .\ar[ddl]^{\b_2}\\
&  & & & &&&&&&&&&&&&&&\\
& .\ar[uul]^{\c_2}  & & . \ar[ll]^{\c_1}  & &&&&&&&&&&& . \ar[luu]^{\b_4} & & . \ar[ll]^{\b_3} &
}$$
Let $I_1$ be the ideal generated by relations of the form $\c_i\c_{i+1}$ and $\b_i\b_{i+1}$ for $i\in\{0, 1,2,3,4\}$ where indices are to be read modulo $5$. Then $\K Q/I_1$ is a $3$-cluster tilted algebra of type $\mathbb{A}_{14}$. On the other hand let $I_2$ be $I_1$ plus the ideal generated by relations $\a_i \a_{i+1}$. The algebra $\K Q /I_2$ is not $3$-cluster tilted, but the two algebras are derived equivalent, as we shall see.  
\end{ex}

We are interested in algebras derived equivalent to $m$-cluster tilted algebras of type $\mathbb{A}$.  A first result is the following.

\begin{prop}\label{prop:HH-branched}
 Let $A=\K Q/I$ be an $m$-cluster tilted algebra of type $\mathbb{A}$, and $B$ an algebra derived equivalent to $A$. Then
\begin{enumerate}
 \item[a)]  The quivers of $A$ and $B$ have the same Euler characteristic, and
 \item[b)] In both cases each cycle is an $(m+2)$-cycle with full relations.
\end{enumerate}
\end{prop}
\begin{proof}\ 
\begin{enumerate}
 \item[$a)$] The bound quivers we are interested in are gentle and have no multiple arrows. In this context, the dimensions of the first Hochschild cohomology groups of $A$ and $B$ are precisely the Euler characteristics of $Q_A$ and $Q_B$. The invariance of these cohomology groups gives the result (compare with remark \ref{rems:core} $a)$).
 \item[$b)$] For the second statement, it is enough to see that remark  \ref{rems:core} $b)$ still holds : the Hochschild cohomology ring structure encodes the number of $(m+2)$-cycles with full relations. 
\end{enumerate}
\end{proof}

Remark that, alternatively, we could have invoked proposition B from \cite{AG08}, which shows that the number of arrows in the quiver of $A=\K Q /I$ is a derived invariant for gentle algebras, and thus so is the Euler characteristic $\chi(Q)$. However, this result is obtained by means of the map $\phi$, and our approach resides in the use of the Hochschild cohomology groups instead.

\medskip

This motivates the following definition.

\begin{de}\label{de:def-branched} We say that an algebra $A=\K Q/I$ is \textbf{branched} (or, more precisely, $m$-branched) if $(Q,I)$ is a  gentle bound  quiver which contains exactly $\chi (Q)$ oriented cycles of length $m+2$, each of them with full relations.
\end{de}

For instance, the algebras of example \ref{ex:branched} are both branched. Also, it follows from  \cite[2.20]{M10} that an $m$-branched algebra $B=\K Q/I$ such that every relation of $(Q,I)$ lies in a cycle of $Q$ is itself an $m$-cluster tilted algebra of type $\mathbb{A}$. 

\medskip The following proposition follows immediately, by reinterpreting  \ref{prop:HH-branched} with the terminology introduced in the previous definition.

\begin{prop}\label{prop:branched-derived}\ 
 \begin{enumerate}
  \item[$a)$] An algebra $B$ derived equivalent to the normal form $N_{r,s}$ is $m$-branched, with $r$ oriented cycles, moreover its quiver has $s$ vertices;
  \item[$b)$] If an algebra $B$ is derived equivalent to an $m$-cluster tilted algebra of type $\mathbb{A}$, then each connected component of $B$ is $m$-branched.
 \end{enumerate}
\end{prop}\qed

One of the main results of this paper says that the converse also holds true, namely that the algebras derived equivalent to $m$-cluster tilted algebras of type $\mathbb{A}$ are precisely the $m$-branched algebras. To prove this it is enough to show that any $m$-branched algebra is derived equivalent to an $m$-cluster tilted algebra of type $\mathbb{A}$. Moreover, since the family of normal forms $N_{r,s}$ contains a representative of each class of derived equivalence of $m$-cluster tilted algebras, it is enough to show that any $m$-branched algebra is derived equivalent to an algebra $N_{r,s}$.  For this, we introduce a somehow intermediate class, see \cite[2.20]{M10}.

\begin{de} A branched algebra $A= \K Q/I$ is said to be \emph{$\mathbb{A}$-branched} (or $\mathbb{A}$-branched  with $(m+2)$-cycles to be more precise) if every relation of $(Q,I)$ lies in a (oriented) cycle of $Q$. 
\end{de}

Note that since a branched algebra is gentle, then being $\mathbb{A}$-branched implies that the quiver obtained from $Q$ by deleting all the arrows in each $(m+2)$-cycle is the union of quivers of type $\mathbb{A}$.  As an example, every normal form $N_{r,s}$, is of this type, and in addition, it follows from \cite[Section 2]{M10} that if $A$ is $\mathbb{A}$-branched, then it is a connected component of an $m$-cluster tilted algebra. We thus have:
\begin{lem}\label{lem:branched-derived-normal} Let $A=\K Q/I$ be an $\mathbb{A}$-branched algebra. Then it is derived equivalent to a normal form $N_{r,s}$.
 \end{lem}
\qed

In order to complete the characterization, it remains to show that  every branched algebra is derived equivalent to an $\mathbb{A}$-branched algebra,  that is, that we can eliminate the relations that lie outside the cycles. This is done in section \ref{sec:remove-relations}, but before that we need  preliminary constructions.


\section{Brenner  - Butler tilting modules}
Let $(Q,I)$ be a gentle bound quiver without loops and $x\in Q_0$ such that whenever there is an arrow leaving $x$, say \mor{\a}{x}{?} then there is an arrow entering $x$, say \mor{\b}{?}{x} such that $\b \a\not\in I$. This includes for instance the vertices that are not the source of any arrow, but excludes the sources of $Q$. The vertex $x$ has at most two arrows leaving it, say $\a_0$ and $\a_1$, and let $\b_0, \b_1$ be the arrows such that $\b_i\a_i \not\in I$, for $i\in\{0,1\}$. In addition, for each $i$, there exists at most one arrow $\c_{i+1}$ such that $\c_{i+1} \b_i \in I$. Note that, reading indices modulo $2$, since  the algebra is gentle, we have $\a_i \b_{i+1},\ \c_{i+1}\b_i \in I$.

In \cite{AH81} Assem and Happel showed that an algebra whose quiver is a gentle tree is tilting-cotilting equivalent to an hereditary algebra of type $\mathbb{A}$. This had been done by explicitly giving a sequence of tilting and cotilting modules. At each stage the gentle tree is transformed until the quiver $\mathbb{A}$ is reached. We will exhibit an analogous process, called \emph{``elementary transformation over a vertex''} in \cite[Section 7]{AA07}, see also \cite[Section 2]{BB10}.

\begin{de}
Let $(Q,I)$ be a gentle bound quiver, and $x$ as above. With these notations the bound quiver obtained by \emph{mutating} $(Q,I)$ at $x$ is the bound quiver defined by $(Q',I')=\sigma_x(Q,I)$ where:
\begin{itemize}
 \item $Q'_0 =Q_0$,
 \item $Q'_1 = Q_ 1 \backslash\{\a_i,\b_i,\c_i|i=0,1\}\cup \{\a'_i,\b'_i,\c'_i|i=0,1\}$ such that \mor{\a'_i}{b_i}{a_i}, \mor{\b'_i}{x}{b_i}, \mor{\c'_i}{c_i}{x}.
\end{itemize}
Let $\mathcal{R}$ be a minimal set of relations generating $I$, in particular it contains $\a_i \b_{i+1},\ \b_i\c_{i+1}$. Let $\mathcal{R}'$ be  obtained by replacing in $\mathcal{R}$ the latter by  $\b'_{i+1}\a'_i$, $\c'_{i+1}\b'_i$ for $i=0,1$, and, again, indices are to be  read modulo 2.
\end{de}

 In the sequel, a dotted line joining two arrows 
means, as usual, that their composition belongs to $\mathcal{R}$.

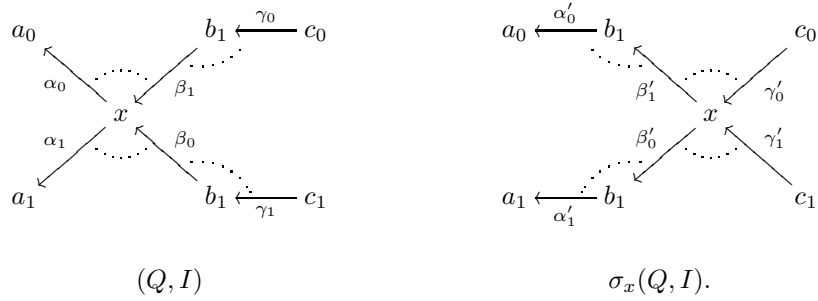
\begin{center} 

 \begin{tabular}{ccccccc}
 $\SelectTips{eu}{10}\xymatrix@C=.3pc@R=.2pc{    a_0  & && & & & b_1  \ar[dddlll]^{\b_1}  & \ar@/^/@{.}[lld]&& c_0  \ar[lll]_{\c_0} \\
&&&&&&&&&\\
&& \ar@/^/@{.}[rr]&&&&&&& \\
&&& x
 \ar[uuulll]^{\a_0} \ar[dddlll]_{\a_1} &&&&&&\\
&&\ar@/_/@{.}[rr]&&&&&&& \\
&&&&&&&&&\\
  a_1  & && & & & b_1  \ar[uuulll]_{\b_0} &\ar@/_/@{.}[llu] && c_1  \ar[lll]^{\c_1} }$&  &&& & & $\SelectTips{eu}{10}\xymatrix@C=.3pc@R=.2pc{  a_0 && \ar@/_/@{.}[rrd]& b_1   \ar[lll]_{\a_0'} & && & & & c_0  \ar[dddlll]^{\c_0'} \\
&&&&&&&&&\\
&&&&& \ar@/^/@{.}[rr]&&&& \\
&&&&&& x \ar[uuulll]^{\b_1'} \ar[dddlll]_{\b_0'} &&&\\
&&&&&\ar@/_/@{.}[rr]&&&& \\
&&&&&&&&\\
 a_1 & &\ar@/^/@{.}[rru] & b_1  \ar[lll]^{\a_1'} & && & & & c_1  \ar[uuulll]_{\c_1'}}$\\
 &  & &&& &\\
$(Q,I)$ & &&& & & $\s_x(Q,I)$.
   \end{tabular}
\end{center}
\medskip 

In the setting above, the almost complete tilting module $\overline{T}=\underset{y\not=x}{\bigoplus} P_y$ is sincere, so by \cite{HU89} there are exactly two indecomposables  $X$ such that $\overline{T}\oplus X$ is a tilting module. One of them is $P_x$. The other is the cokernel $X'$ of the left minimal $(\ts{add}-\overline{T})$-approximation  $P_x \to P_{c_0}\oplus P_{c_1}$ of $P_x$. One can compute (see \cite[II]{ASS06}, for instance) that $X'\simeq \tau^{-1}S_x$. This leads to the known result:

\begin{lem}[\cite{BB80}] Let $A=\K Q/I$ be a gentle algebra and $x\in Q_0$ as above. Then
\begin{enumerate}
 \item[$a)$] The module $T_x = \tau^{-1} S_x \oplus \overline{T}=\underset{y\not=x}{\bigoplus} P_y$ is a tilting $A$-module;
 \item[$b)$] The quiver of $\End{A}{T_x}$ is precisely $\sigma_x(Q,I)$.
\end{enumerate}
\end{lem}

\begin{proof} We have already proved statement $a)$, while statement $b)$ is a straightforward computation.

\end{proof}

The tilting module $T_x$ is called the \emph{Brenner - Butler tilting module at $x$}, or BB tilting module, for short. In an analogous way one can define the \emph{BB cotilting module at a vertex $y$}, and the corresponding mutation $\s'_y$ on the bound quivers.

\begin{rem} \label{rem:M1-Tilting-derived} 
From \cite[3.13, 3.16 and 3.18]{M10} and the previous lemma,  it follows that two $m$-cluster tilted algebras  of type $\mathbb{A}$ are derived equivalent if and only if they are tilting-cotilting equivalent. See also \cite{BB10} for the case $m=1$. In fact, the \emph{local mutations} induced by the so-called \emph{elementary polygonal moves} of \cite{M10} are tilting-cotilting equivalences.
\end{rem}

\section{Removing relations}\label{sec:remove-relations}
In light of the previous sections, the only thing that remains to do is to show that a branched algebra $A=\K Q/I$ that is not $\mathbb{A}$-branched is tilting-cotilting equivalent to an algebra $A'=\K Q' / I'$ such that $(Q',I')$ is $\mathbb{A}$-branched. To do so, we apply a sequence of mutations that remove the relations of $(Q,I)$ which do not lie in an oriented cycle. For this, we need to establish an ordering of the relations that are to be removed.

Let $(Q,I)$ be a branched quiver, and  $\SelectTips{eu}{10}\xymatrix@1@C=7pt{x\ar[rr] & \ar@{.}@/^/[rr]^\rho & y \ar[rr] & & z}$ a relation not lying in a cycle. The quiver obtained from $Q$ by removing the vertex $y$ and all its incident arrows is disconnected. Denote by $Q_{\rho}^-$ the connected component containing $x$, $Q_{\rho}^+$ that containing $z$ and $Q_\rho^\perp$ the (possibly empty) union of the remaining connected components.

$$\SelectTips{eu}{10}
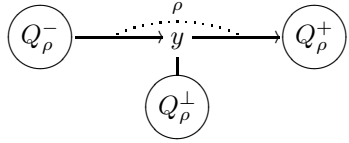
\xymatrix@R=.2pc@C=.8pc{ *++[o][F]{Q_\rho^-}  \ar[rr] & \ar@{.}@/^/[rr]^{\rho} & y \ar[rr]  & &   *++[o][F]{Q_\rho^+} \\
 && *++[o][F]{Q_\rho^\perp}\ar@{-}[u] && }$$

In the sequel we adopt the following convention concerning decorations on the names of quivers: $\SelectTips{eu}{10}\xymatrix@R=.2pc@C=.8pc{ *+[o][F]{Q}}$ means that the quiver $Q$ is bound by an ideal $I$ such that $(Q,I)$ is a branched bound quiver, whereas $\SelectTips{eu}{10}\xymatrix@R=.2pc@C=.8pc{ *+[F]{Q}}$ means that the latter is $\mathbb{A}$-branched.
 
 \begin{de} Let $(Q,I)$ be a branched bound quiver, and   $\SelectTips{eu}{10}\xymatrix@1@C=7pt{x\ar[rr] & \ar@{.}@/^/[rr]^\rho & y \ar[rr] & & z}$ a relation not lying in a cycle. The relation $\rho$ is said to be:
\begin{enumerate}
 \item[$a)$] \emph{$-$extremal} if $Q_\rho^-$ is an $\mathbb{A}$-branched quiver,
 \item[$b)$] \emph{$+$extremal} if $Q_\rho^+$ is an $\mathbb{A}$-branched quiver.
\end{enumerate}
\begin{center}
 \begin{tabular}{ccccc}
$\SelectTips{eu}{10}\xymatrix@R=.2pc@C=.8pc{ *++[F]{Q_\rho^-}  \ar[rr] & \ar@{.}@/^/[rr]^{\rho} & y \ar[rr]  & &   *++[o][F]{Q_\rho^+} \\
 && *++[o][F]{Q_\rho^\perp}\ar@{-}[u] && }$ & && &$\SelectTips{eu}{10}\xymatrix@R=.2pc@C=.8pc{ *++[o][F]{Q_\rho^-}  \ar[rr] & \ar@{.}@/^/[rr]^{\rho} & y \ar[rr]  & &   *++[F]{Q_\rho^+} \\
 && *++[o][F]{Q_\rho^\perp}\ar@{-}[u] && }$ \\
 &  & &&\\
a $-$extremal relation & & & & a $+$extremal relation.
   \end{tabular}
\end{center}
 \end{de}

\begin{lem} Le $A=\K Q/I$ be a branched quiver which is not $\mathbb{A}$-branched. Then $(Q,I)$ has at least one $-$extremal relation and at least one $+$extremal relation.
\end{lem}
\begin{proof}
We proceed inductively on $|Q_0|$.  Let $(Q,I)$ be a branched quiver, in the sense of \ref{de:def-branched}, which is not $\mathbb{A}$-branched, and $\rho$  a relation not lying in a cycle which in addition we assume not to be $+$extremal. That means that the bound quiver $Q_\rho^+$ has a relation outside its cycles. But since $\left| \left(Q_\rho^+\right)_0\right| < |Q_0|$, the quiver $Q_\rho^+$ has a $+$extremal relation, which is seen to be a $+$extremal relation of $Q$.  
\end{proof}

We are now able to describe the sequence of mutations that  allows to remove the relations outside the cycles.

\begin{lem}\label{lem:no-relations}
The  gentle bound quiver
$$\SelectTips{eu}{10}
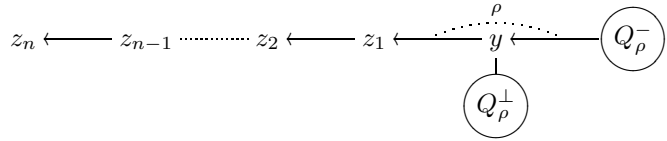
\xymatrix@R=.1pc@C=.8pc{ z_n &&  z_{n-1} \ar[ll] &&  z_2 \ar@{.}[ll]  &&   z_1 \ar[ll] & \ar@{.}@/^/[rr]^{\rho} & y \ar[ll]  & &*++[o][F]{Q_\rho^-}  \ar[ll]\\
 &&&&&&&& *++[o][F]{Q_\rho^\perp}\ar@{-}[u]  && }$$
is tilting-cotilting equivalent to the gentle bound quiver
$$\SelectTips{eu}{10}
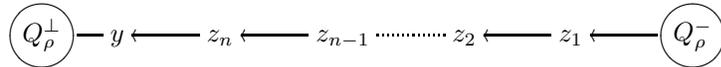
\xymatrix@R=.1pc@C=.8pc{ *++[o][F]{Q_\rho^\perp}\ar@{-}[r]&  y  && z_n \ar[ll] &&  z_{n-1} \ar[ll] &&  z_2 \ar@{.}[ll]  &&   z_1 \ar[ll] &&   *++[o][F]{Q_\rho^-}  \ar[ll]}$$
\end{lem}
\begin{proof}Apply the sequence of mutations $\sigma_{z_n} \sigma_{z_{n-1}} \cdots \sigma_{z_1}$.
 \end{proof}

If we are given a branched algebra $A=\K Q/I$ and a $+$extremal relation $\rho$ such that $Q^+_\rho$ is a quiver of type $\mathbb{A}$, the previous lemma shows how to obtain an algebra $A'=\K Q' / I'$ which is tilting-cotilting equivalent to $A$. Moreover in $(Q',I')$ there is one relation less than in $(Q,I)$, namely the relation $\rho$. Thus, it only remains to see how to remove a $+$extremal relation $\rho$ such that $Q^+_\rho$ has cycles. Moreover, since we assume that the only relations in $Q_\rho^+$ lie on its cycles, the algebra whose (bound) quiver is $Q_\rho^+$ is, by the results of Murphy \cite{M10}, derived equivalent to a normal form $N_{r,s}$.

\begin{lem}\label{lem:with-relations}
The algebra given by the  quiver with relations:

$$\SelectTips{eu}{10}
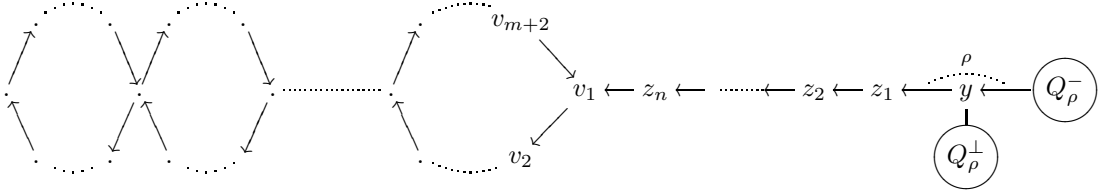
\xymatrix@C=.2pc@R=.1pc{   & .\ar@{.}@/^/[rrr] &&& .\ar[ddr] &  & . \ar@{.}@/^/[rrr] &&& .\ar[ddr] &  &  &&&&  & . \ar@{.}@/^/[rrr] &&& v_{m+2} \ar[ddr] &  &&&&&&&&&&\\
&&&&&&&&&&&&&&&&&&&&  &&&&&&&& &&\\
. \ar[ruu] & &&& & . \ar[dl] \ar[uur] &  &&&  & . \ar[dl] \ar@{.}[rrrrr] & &&&&  .\ar[uur] & &&& & v_1 \ar[dl]  && z_n \ar[ll]  &&  \ar[ll]\ar@{.}[rrrr] &&&& z_2 \ar[ll]&& z_1 \ar[ll] & \ar@{.}@/^/[rr]^\rho & y  \ar[ll] &&  *++[o][F]{Q_\rho^-} \ar[ll]  \\
& .\ar@{.}@/_/[rrr]  \ar[ul] &&& . &  & . \ar@{.}@/_/[rrr] \ar[ul] &&&  &  &  &&&&  & . \ar@{.}@/_/[rrr] \ar[ul] &&& v_2  &    &&&&&&&&&&&& *++[o][F]{Q_\rho^\perp}\ar@{-}[u] && 
}$$

 is tilting-cotilting equivalent to the  algebra given by the  quiver with relations:

 $$\SelectTips{eu}{10}
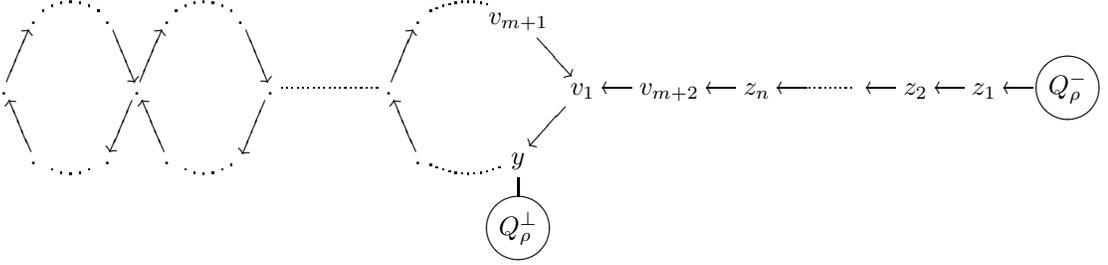
\xymatrix@C=.2pc@R=.1pc{   & .\ar@{.}@/^/[rrr] &&& .\ar[ddr] &  & . \ar@{.}@/^/[rrr] &&& .\ar[ddr] &  &  &&&&  & . \ar@{.}@/^/[rrr] &&& v_{m+1} \ar[ddr] &  &&&&&&&&&&\\
&&&&&&&&&&&&&&&&&&&&  &&&&&&&&&&\\
. \ar[ruu] & &&& & . \ar[ddl] \ar[uur] &  &&&  & . \ar[ddl] \ar@{.}[rrrrr] & &&&&  .\ar[uur] & &&& & v_1 \ar[ddl]  && v_{m+2} \ar[ll]  &&  z_{n} \ar[ll]\ar@{.}[rrrr]  && \ar[ll] &&&&  z_2 \ar[ll]&& z_1 \ar[ll] & &  *++[o][F]{Q_\rho^-}\ar[ll]  \\
&&&&&&&&&&&&&&&&&&&&  &&&&&&&&&& \\
& .\ar@{.}@/_/[rrr]  \ar[uul] &&& . &  & . \ar@{.}@/_/[rrr] \ar[uul] &&&  &  &  &&&&  & . \ar@{.}@/_/[rrr] \ar[uul] &&& y   &    &&&&&&&&&& \\
&&&&&&&&&&&&&&&&&&& \\
&&&&&&&&&&&&&&&&&&& *++[o][F]{Q_\rho^\perp}\ar@{-}[uu]}$$
 
\end{lem}

\begin{proof}
 As in the previous lemma, use the sequence $\sigma_{z_n} \sigma_{z_{n-1}} \cdots \sigma_{z_1}$ to take $Q_\rho^\perp$ to the vertex $v_1$, then apply $\sigma_{v_{m+2}} \sigma_{v_1}$.
\end{proof}

Thus, we know how to eliminate a $+$extremal relation with mutations. Dually we can eliminate $-$extremal relations. We have the following

\begin{prop}\label{prop:branched-Abranched} Let $A=\K Q/I$ be a branched algebra, then there exists an $\mathbb{A}$-branched algebra $A'$ such that $A$ and $A'$ are tilting-cotilting equivalent.
\end{prop}

\begin{proof} Let $A=\K Q/I$ be a branched algebra which is not $\mathbb{A}$-branched. The result follows upon executing the following: \newline
 
{\bf Algorithm. } \begin{description}
 \item[Step 1] Choose  an extremal relation $\rho$ in $(Q,I)$.
 \item[Step 2] Using the processes described in lemmata \ref{lem:no-relations} and \ref{lem:with-relations} obtain a new algebra $A'=\K Q'/I'$ which is tilting-cotilting equivalent to $A$. In $(Q',I')$ there is one relation less than in $(Q,I)$, namely the relation $\rho$.
 \item[Step 3] If $A'$ is $\mathbb{A}$-branched we are done. If not, go to step 1 with $A'$ playing the r\^ole of $A$.
\end{description}
This process must stop after a finite number of steps, since the number of relations outside the cycles of $(Q,I)$ is finite.
\end{proof}
 \begin{cor}  Let $A=\K Q/I$ be an $m$-branched algebra, then $A$ is tilting-cotilting equivalent to an $m$-cluster tilted algebra of type $\mathbb{A}$.
   \end{cor}

\begin{proof} The previous result says that $A$ is tilting-cotilting equivalent to an $\mathbb{A}$-branched algebra $A'$. By Lemma \ref{lem:branched-derived-normal} we know that $A'$ is tilting-cotilting equivalent to a suitable normal form $N_{r,s}$, which is $m$-cluster tilted.
 \end{proof}

\section{Proof of the main theorem}
We are now able to state and prove the first main theorem of this paper. Observe that Theorem A is precisely the equivalence between conditions $a)$ and $b)$. 

\begin{thm}\label{thm:main}
 Let $A= \K Q/I$ be a connected algebra. Then the following conditions are equivalent.
\begin{enumerate}
 \item[$a)$] $A$ is an $m$-branched algebra,
 \item[$b)$] $A$ is derived equivalent to an $m$-cluster tilted algebra of type $\mathbb{A}$,
 \item[$c)$] $A$ is tilting-cotilting equivalent to an algebra $N_{r,s}$,
 \item[$d)$] $A$ is tilting-cotilting equivalent to an $m$-cluster tilted algebra of type $\mathbb{A}$.
\end{enumerate}
\end{thm}
\begin{proof} \
\begin{enumerate}
 \item[\ ]$a)$ implies $d)$. This is proposition  \ref{prop:branched-Abranched}, since $\mathbb{A}$-branched algebras are themselves (connected components of) $m$-cluster tilted algebras of type $\mathbb{A}$.
 \item[]$d)$ implies $c)$. By Murphy's results, an $m$-cluster tilted algebra is derived equivalent to some normal form $N_{r,s}$, which is itself $m$-cluster tilted. The result then follows from \ref{rem:M1-Tilting-derived}.
 \item[]$c)$ implies $b)$. This is immediate
 \item[]$b)$ implies $a)$. This follows directly from statement $a)$ in proposition \ref{prop:branched-derived}.
\end{enumerate}
\end{proof}

In light of the preceding result, if a branched algebra $A$ is tilting - cotilting equivalent to the normal form $N_{r,s}$, we call the pair $(r,s)$ the \emph{invariant pair} of $A$. It follows from proposition \ref{prop:branched-derived} that the invariant pair of $A$ completely characterizes the class of algebras derived equivalent to it. 

As a first step towards the proof of Theorem B, we have:

\begin{cor} Let $A=\K Q/I$ and $A'=\K Q'/I'$ be connected  $m$-branched algebras. Then the following conditions are equivalent:
\begin{enumerate}
 \item[$a)$] $A$ and $A'$ are derived equivalent,
\item[$b)$] $A$ and $A'$ are tilting-cotilting equivalent,
\item[$c)$] $A$ and $A'$ have the same invariant pair,
\item[$d)$] $\HH{\ast}{A} \simeq \HH{\ast}{A'}$ and $K_0(A) \simeq K_0(A')$,
\item[$e)$] $\pi_1(Q,I) \simeq \pi_1(Q',I')$ and $|Q_0|  = |Q'_0|$.
\end{enumerate}
\end{cor}
 
\begin{proof} The equivalence of $a)$ and $b)$ has been established in \ref{rem:M1-Tilting-derived}. The fact that $a)$ and $b)$ imply $d)$ is the invariance of the Hochschild cohomology ring under derived equivalence, together with that of the Grothendieck group $K_0(A)$. Since $\ts{rk\ }K_0(A)$ equals $|Q_0|$  and $\dim{\HH{1}{A}}=\chi(Q)$ we have that  $d)$ implies $c)$, and $d)$ is equivalent to $f)$. Finally, $c)$ implies $b)$, by \ref{thm:main}.
\end{proof}

Note that, in addition, proposition \ref{prop:HH-branched} gives an explicit description of the structure of $\HH{\ast}{A}$ in this case. Furthermore, the following example shows that the connectedness hypothesis cannot be dropped.

\begin{ex}
 Consider the two disconnected bound quivers depicted above. Recall that each cycle is to be understood as a cycle with full relations.
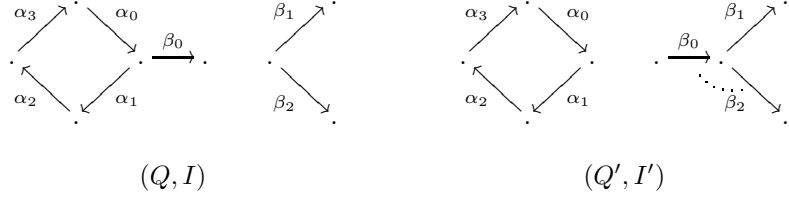
\begin{center}
 \begin{tabular}{ccccc}
$\SelectTips{eu}{10}\xymatrix@R=.4pc@C=.4pc{
			&& .\ar[ddrr]^{\a_0}	&&					&&	&&	&	&.	\\
			 &&				&&					&&	&&	&	&	\\
 .\ar[uurr]^{\a_3}	&&				&&.\ar[ddll]^{\a_1}\ar[rr]^{\b_0}	&&.	&&.\ar[uurr]^{\b_1}\ar[ddrr]_{\b_2}	&	&	\\
			 &&				&&					&&	&&	&	&	\\
			&&.\ar[uull]^{\a_2}		&&					&&	&&	&	&.	}
$&  & & &$\SelectTips{eu}{10}\xymatrix@R=.4pc@C=.4pc{
			&& .\ar[ddrr]^{\a_0}	&&					&&	&&	&	&.	\\
			 &&				&&					&&	&&	&	&	\\
 .\ar[uurr]^{\a_3}	&&				&&.\ar[ddll]^{\a_1}	&&.\ar[rr]^{\b_0}	&\ar@/_/@{.}[drr]&.\ar[uurr]^{\b_1}\ar[ddrr]_{\b_2}	&	&	\\
			 &&				&&					&&	&&	&	&	\\
			&&.\ar[uull]^{\a_2}		&&					&&	&&	&	&.	}
$\\
&&&&\\
$(Q,I)$&&&&$(Q',I')$\\
&&&&
\end{tabular}
\end{center}
The two algebras have isomorphic Hochschild cohomology rings, isomorphic Grothendieck groups and the same invariant pair, but are not derived equivalent.
\end{ex}

Specializing to the case $m=1$ we obtain a characterization of the cluster tilted algebras of type $\mathbb{A}$, and an explicit description of the corresponding bound quivers, which in addition is very easy to verify.  Compare with \cite[Theorem A]{BB10}, and \ref{cor:ThmA-BB10}.

\begin{cor}
 A connected gentle algebra $A=\K Q/I$ is derived equivalent to a cluster tilted algebra of type $\mathbb{A}$ if and only if $(Q,I)$ has $\chi(Q)$ oriented cycles of length $3$, each of which has full relations. 
\end{cor}
\qed

\begin{ex} Remember that, as usual, doted lines mean relations, and that each oriented $3$-cycle has full relations. Let $A$ be  algebra whose bound quiver is depicted below  $$\SelectTips{eu}{10}
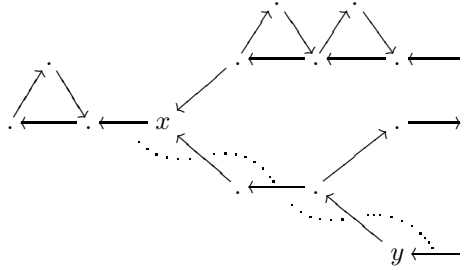
\xymatrix@C=.5pc@R=.4pc{ 
		&		&		&&		&			&			&.\ar[ddr]		&			&.\ar[ddr]	&		&	&\\
		&		&		&&		&			&			&			&			&		&		&	&\\
		&.\ar[ddr]	&		&&		&			&.\ar[uur]\ar[ddll]	&			&.\ar[ll]\ar[uur]	&		&.\ar[ll]	&	&.\ar[ll]\\
		&		&		&&		&			&			&			&			&		&		&	&\\
.\ar[uur]	&		&.\ar[ll]	&&x\ar[ll]	&			&			&			&			&		&.\ar[rr]	&	&.\\
		&		&		&&		&\ar@{.}@/^/[ull]	&			&			&			&		&		&	&\\
		&		&		&&		&			&.\ar[uull]		&\ar@{.}@/_/[ull]	&.\ar[ll]\ar[uurr]	&		&		&	&\\
		&		&		&&		&			&			&			&			&\ar@{.}@/^/[ull]		&		&	&\\
		&		&		&&		&			&			&			&			&		&y\ar[uull]	&\ar@{.}@/_/[ull]	&.\ar[ll]
}$$

$A$ is not cluster tilted, because there are relations outside the $3$-cycles, so that its Gorenstein dimension is $3$ (see \cite{GR05}). However, it is tilting-cotilting equivalent to a cluster tilted algebra of type $\mathbb{A}_{16}$. Its invariant pair is $(3,16)$. The relation involving vertex $x$ is $+$extremal; that involving $y$ is $-$extremal.

\end{ex}

Recall that a representation finite algebra is said to be \emph{simply connected} whenever its Auslander -- Reiten quiver is simply connected, as a 2-dimensional simplicial complex, or, equivalently, whenever it has no proper Galois covering (see \cite{AdlP96}, and the references therein). We have the following: 

\begin{cor} Let $A$ be a connected component of an $m$-branched algebra, then the following conditions are equivalent.
\begin{enumerate}
 \item[$a)$] $A$ is simply connected,
 \item[$b)$] $A$ is iterated tilted of type $\mathbb{A}$,
 \item[$c)$] $A$ is representation directed,
 \item[$d)$] $\HH{1}{A}=0$.
\end{enumerate} 
\end{cor}

\begin{proof} This follows from the fact that $m$-branched algebras are monomial, and $\dim{\HH{1}{A}} = \chi(Q)$, regardless  the characteristic of the field $\K$.
 
\end{proof}

\begin{rem} Proposition \ref{prop:HH} shows that  $\HH{2}{A} =0$, whenever $A$  is an $m$-branched algebra, this entails that $m$-branched algebras are rigid.
\end{rem}

\section{Further consequences and remarks}
\subsection{The Avella Alaminos - Geiss map $\phi$}\label{subsec:AG-phi}
As mentioned  before, in \cite{BB10} the map $\phi$, which is a derived invariant for gentle algebras \cite{AG08}, is the core tool used to establish the derived equivalence classification therein. It is a map \mor{\phi}{\mathbb{N} \times \mathbb{N} }{\mathbb{N}} that counts special sequences of paths and relations  in a gentle quiver  $(Q,I)$. We are interested in characterizing this map for normal forms $N_{r,s}$. In what follows we closely follow the exposition of \cite{AG08}.

Let $A=\K Q/I$ be a gentle algebra. A \emph{permitted thread} of $A$ is a path $w=\a_1 \a_2\cdots \a_n$ not belonging to $I$, and of maximal length for this property. A \emph{forbidden thread} is an element of some $\G_n$, with maximal length and without repeated arrows.

We also need trivial permitted and forbidden threads. Let $x\in Q_0$ be such that the sets $s^{-1}(x)$ and $t^{-1}(x)$ have both cardinality at most one. The stationary path at $x$ is a \emph{trivial permitted thread} if when   \mor{\a}{?}{x} and \mor{\b}{x}{?} are two arrows, then $\a\b\not\in I$. We denote this thread by $h_x$. Similarly, the stationary path at $x$ is a \emph{trivial forbidden thread} if when  \mor{\a}{?}{x} and \mor{\b}{x}{?} are two arrows, then $\a\b\in I$. We denote by $p_x$ this thread. Assume $x$  and $y$ are two vertices such there is only one one arrow $\a$ entering $x$, an arrow \mor{\b}{x}{y}, and only one arrow $\c$ leaving $y$. If both $\a\b$, $\b\c\in \mathcal{R}$, then $\b$ is a permitted thread, whereas in case $\a\b$, $\b\c\not\in \mathcal{R}$ the arrow $\b$ is a forbidden thread.

Given that $(Q,I)$ is gentle, from \cite{BR87} one knows that there exist maps \mor{\s , \e}{Q_1}{\{\pm 1\}} verifying:
\begin{itemize}
 \item[$\cdot$] $\s(\b_0) = -\s(\b_1)$ whenever $\b_0$ and $\b_1$ are arrows sharing their source; 
 \item[$\cdot$] $\e(\b_0) = -\e(\b_1)$ whenever $\b_0$ and $\b_1$ are arrows sharing their target;
 \item[$\cdot$] If $\a \b$ is a path not belonging to $I$, then $\s(\b)= - \e(\a)$.
\end{itemize}
These maps, that one can set \emph{``quite arbitrarily''}, as noted in \cite[p. 158]{BR87},  extend to paths, thus to threads: given $w=\a_1 \a_2\cdots \a_n$, set $\s(w)=\s(\a_1)$ and $\e(w)=\e(\a_n)$. We extend this to trivial threads as follows: If $h_x$ is a trivial permitted thread, let \mor{\b}{x}{?} be the (unique) arrow leaving $x$. Then put $\s(h_x)=-\e(h_x)=-\s(\b)$. Similarly, if $p_x$ is the trivial forbidden thread at $x$, let \mor{\a}{?}{x} be the (unique) arrow ending at $x$. Then, put $\s(p_x)=-\e(p_y) = -\e(\a)$. Given a path $w$, denote by $\ell(w)$ its length, that is its number of arrows.
\setcounter{thm}{1}
\begin{algo}[Avella - Alaminos and Geiss \cite {AG08}]\ Let $A=\K Q/I$ be a gentle bound quiver, for which all permitted and forbidden threads are determined.
\begin{enumerate}
 \item \begin{enumerate}
        \item Begin with a permitted thread $H_0$ of $A$,
        \item If $H_i$ is defined, let $P_i$ be the forbidden thread sharing target with $H_i$ and such that $\e(H_i)=-\e(P_i)$,
        \item Let $H_{i+1}$ be the permitted thread sharing source with $P_i$ and such that $\s(H_{i+1}) = -\s(P_i)$.
        \item[]The process stops if $H_n= H_0$ for some natural number $n$. In this case, let $m = \sum_{1\leqslant i \leqslant n} \ell(P_i)$  
       \end{enumerate}
\item Repeat step 1 until all permitted threads of $A$ have been considered;
\item If there are (oriented) cycles $w$ with full relations, add a pair $(0,\ell(w))$ for each of those cycles;
\item Define \mor{\phi_A}{\mathbb{N} \times \mathbb{N} }{\mathbb{N}} by letting $\phi_A(n,m)$ be the number of times the pair $(n,m)$ appears in the algorithm.
\end{enumerate}
\end{algo}

Theorem A in \cite{AG08} asserts that $\phi$ is a derived invariant, but the converse is only known to be true in some particular cases, see \cite[Theorem C]{AG08} or \cite{AA07}. Also, in \cite[Theorem D]{BB10} the authors showed that two gentle algebras $A$ and $A'$ that are derived equivalent to cluster tilted algebras of type $\mathbb{A}$ or $\tilde{\mathbb{A}}$, are themselves derived equivalent if and only if $\phi_A = \phi_{A'}$. We shall state and prove an analogous result. For this sake, we need to compute the function $\phi_{r,s}$ corresponding to a normal form $N_{r,s}$.  As a first remark, note that $\phi_{r,s} (0,m+2)$, is the number of $(m+2)$-cycles with full relations in $(Q,I)$, that is,  $\phi_{r,s} (0,m+2)= r$, see \cite[Remark 6]{AG08}. To compute the remaining part of $\phi_{r,s}$,  label the arrows of $Q_{r,s}$ as follows: the cycles are numbered as in the proof of \ref{prop:HH}, and again $v_j$, be the vertex common the $j^{th}$ and the $(j+1)^{th}$ cycles. Furthermore, let $1$ be the unique vertex of degree $3$ in $Q_{r,s}$. For each $j\in \{1,2\ldots ,r\}$ let $\a^j_0, \a^j_2,\ldots , \a^j_{m+1}$ be the $m+2$  arrows of the $j^{th}$ cycle, in such a way that the source of $\a^j_0$ is precisely $v_{j-1}$, setting $v_0 =1$. In addition, let $\b_1,\b_2,\ldots , \b_n$ be the remaining arrows, where the source of $\b_n$ is $n+1$, the unique source of the quiver,  and $\b_n \ldots  \b_2\b_1$ does not belong to $I$. See below.

$$\SelectTips{eu}{10}
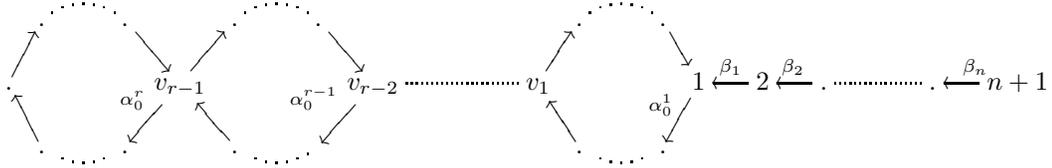
\xymatrix@C=.3pc@R=.3pc{   & .\ar@{.}@/^/[rrr] &&& .\ar[ddr] &  & . \ar@{.}@/^/[rrr] &&& .\ar[ddr] &  &  &&&&  & . \ar@{.}@/^/[rrr] &&& . \ar[ddr] &  &&&&&&&&&&\\
&&&&&&&&&&&&&&&&&&&&  &&&&&&&&&&\\
. \ar[ruu] & &&& & v_{r-1} \ar[ddl]_{\a^r_0} \ar[uur]&  &&&  & v_{r-2} \ar[ddl]_{\a^{r-1}_0} \ar@{.}[rrrrr] & &&&&  v_{1}\ar[uur] & &&& & 1\ar[ddl]_{\a^1_0}  && 2\ar[ll]_{\b_1}  && . \ar@{.}[rrrr] \ar[ll]_{\b_{2}} &&&& . && n+1\ar[ll]_{\b_n}\\
&&&&&&&&&&&&&&&&&&&&  &&&&&&&&&& \\
& .\ar@{.}@/_/[rrr]  \ar[uul] &&& . &  & . \ar@{.}@/_/[rrr] \ar[uul] &&&  &  &  &&&&  & . \ar@{.}@/_/[rrr] \ar[uul] &&& .  &    &&&&&&&&&&}$$

Given a pair $(a,b)\in \mathbb{N} \times \mathbb{N}$, denote by $(a,b)^\ast$ the characteristic function of the set  $\{(a,b)\} \subseteq \mathbb{N} \times \mathbb{N}$. The following will be useful.

\begin{lem}\label{lem:phi-normal} Let $N_{r,s}$ be a normal form, as defined at the end of Section 1. Then $$\phi_{r,s} = r \cdot (0,m+2)^\ast + (s+1-r,s-1-r(m+1))^\ast$$
\end{lem}
\begin{proof} The term $r \cdot (0,m+2)^\ast$ comes from step $(3)$. Before going on, note that counting vertices and arrows yields the relation 
\begin{equation}
 s=n+1+r(m+1).
\end{equation}
From this, we have $s+1-r = n + rm +1$, and $s-1-r(m+1) = n$. We now proceed by induction on $r$. As we find it more illustrative and useful for the sequel, we start with the case $r=2$, leaving the (easier) case $r=1$ to the reader.

Assume $r=2$, and let $k$ be such that $t(\a^1_k)= v_1$. Start the algorithm with $H_0 = \b_n\cdots \b_1\a^1_0$ so that $P_0=p_{t(\a^1_0)}$. Thus, $H_1 = \a^1_1$ and $P_1=p_{t(\a^1_1)}$. Then we continue running through the rightmost cycle, until $H_{k-1}=\a^1_{k-1}$ and $P_{k-1} = p_{t\a^1_{k-1}}$. At the next step we get $H_{k}=\a^1_k\a^2_0$, switching to second cycle, and $P_k= p_{t(\a^2_0)}$. Then, start running through the second cycle, and get $H_{k+m}= \a^2_m$, so that at the next step we have $H_{k+m+1} = \a^2_{m+1}\a^1_{k+1}$, coming back to the first cycle. Running through it, we are led to $H_{2m} = \a^1_m$ and then $H_{2m+1}=\a^1_{m+1}$, but, in contrast with what precedes, $P_{2m+1} = \b_1$. From then, $H_{2m+2}=h_2,\ P_{2m+2}=\b_2$, and we continue going backwards through the arrows $\b$. We arrive then to $H_{2m+n+1} = h_{n+1},\ P_{2m+n+1}= p_{n+1}$ and $H_{2m+n+2} = H_0$. The only forbidden paths of non zero length are the arrows $\b$. We thus obtain the pair $(2m+n+2,n)$, which from the arithmetic relation $(1)$ is easily seen to be the desired one.

Assume now the statement is true for $r=l-1$, and let us analyze the algorithm for $r=l$. The forbidden paths of nonzero length that will appear are exactly the same as those in the case $r=2$, namely the arrows $\b$, so we will get a pair of the form $(T ,n)$. What remains to see is how many steps are added to the algorithm, so we can determine $T$. We claim that there are exactly $m$ of them. Indeed, we are adding a single cycle of length $m+2$. As in the case $r=2$, two of these arrows are used to switch from one cycle to another, and the remaining ones lead (each of them) to one additional step of the algorithm.  The total number of steps is, by the induction hypothesis, and the previous argument
$$T= (n+rm+1)+m = n+(r+1)m+1.$$

Since every permitted thread has been used, this finishes the proof. 

\end{proof}

In particular this result, together with Theorem \ref{thm:main} enables us to recover the classification of the algebras derived equivalent to cluster tilted algebra of type $\mathbb{A}$,  of Bobinski and Buan \cite[Theorem A]{BB10}. Note however that we additionally obtained an explicit description of the bound quiver of such an algebra. The conditions characterizing those bound quivers are  really easy to verify, for no computations are required, as in \cite[Theorem A]{BB10}, which we can easily recover.

\begin{cor}\label{cor:ThmA-BB10}
 Let $A$ be a gentle algebra. Then $A$ is derived equivalent to a cluster tilted algebra of type $\mathbb{A}$ if and only if there exist natural numbers $r,n$ such that $$\phi_A = r\cdot(0,3)^\ast + (n+r+2,n)^\ast$$
\end{cor}

\begin{proof} Assume $A$ is derived equivalent to a cluster tilted algebra of type $\mathbb{A}$. From \ref{thm:main}, $A$ is derived equivalent to a normal form $N_{r,s}$, and $m=1$. Then, use \ref{lem:phi-normal}, and the fact that with those notations one has $s=n+1+2r$.

On the other hand, assume $A=\K Q/I$ is gentle such that  $\phi_A = r\cdot(0,3)^\ast + (n+r+2,n)^\ast$. Since $\phi_A(0,3)=r$, in $(Q,I)$ there are exactly $r$ cycles of length $3$ with full relations. In addition, \cite[6.2]{BB10} gives that the Euler characteristic of $Q$ is precisely $r$, and we are done. 
\end{proof}

\begin{rem} Theorem F in \cite{BB10} asserts that an algebra $A$ which is derived equivalent to a cluster tilted algebra of type $\mathbb{A}$ is itself cluster tilted if and only if its Gorenstein dimension $\ts{Gdim}\ A$ is at most 1. It follows from Murphy's description of $m$-cluster tilted algebras of type $\mathbb{A}$ (see \cite[2.20]{M10}), and \cite[3.4]{GR05} that an $m$-branched algebra is $m$-cluster tilted algebra if and only if $\ts{Gdim}\ A \leqslant m$.
 
\end{rem}

As we noted in the discussion following \ref{subsec:AG-phi}, it is not known in general whether or not $\phi$ is a complete invariant for gentle algebras. It was shown in \cite{AG08} that if $A$ and $A'$ two algebras whose quivers have Euler characteristic $1$ are such that $\phi_A = \phi_{A'}$, then $A$ and $A'$ are derived equivalent, and a similar result in \cite{AA07}. Further, Bobinski and Buan showed in \cite{BB10} that $\phi$ is a complete invariant for cluster tilted algebras. We can prove that this is also the case for $m$-branched algebras. The following completes the proof of Theorem B.

\begin{thm} Let $A$ and $A'$ be $m$-branched algebras. Then $A$ and $A'$ are derived equivalent if and only if $\phi_A = \phi_{A'}$.
\end{thm}

\begin{proof}
Assume tat $A$ and $A'$ are derived equivalent. Then the invariance of $\phi$ and lemma \ref{lem:phi-normal} give the result. 

On the other hand assume  $\phi_A = \phi_{A'}$, and let $N_{r,s}$, $N_{r',s'}$ be the normal forms derived equivalent to $A$ and $A'$, respectively.  Again, it follows from the invariance of $\phi$ that since  $\phi_A=\phi_{r,s}$ and $\phi_{A'}=\phi_{r',s'}$, we have $\phi_{r,s} = \phi_{r',s'}$. It then follows from the previous lemma that $(r,s)=(r',s')$, and thus that $N_{r,s}$ and $N_{r',s'}$ are derived equivalent.
\end{proof}

\setcounter{subsection}{\value{thm}}
\subsection{Cartan matrices}\addtocounter{thm}{1}
In \cite[3.2]{BH07}, Bessenrodt and Holm established that the unimodular equivalence class of the Cartan matrix of a gentle finite dimensional algebra is a derived invariant. This enabled them  to obtain the fact that if two gentle algebras $A=\K Q/I$ and $A'=\K Q' / I'$ are derived equivalent, then the number of cycles of even length with full relations in $(Q,I)$ equals that in $(Q',I')$; and the same holds for cycles with full relations of odd length (see also \cite[Theorem 1]{H05}. This result has been of central importance in the derived classification of Buan and Vatne (see \cite[Section 4]{Buan-Vatne08}), as well in that of Murphy (see \cite[3.8]{M10}). This invariant does not retain the length of the cycles, but only their length modulo 2. In contrast, the Hochschild cohomology ring does detect this information, and, furthermore, it distinguishes the case where the characteristic of the ground field is 2, another feature missed by the Cartan matrix. Thus, the methods provided here can easily be used to obtain independent proofs of the mentioned results of \cite{Buan-Vatne08, M10}. We end with an example.

\begin{ex}
 Let $(Q,I)$ and $(Q',I')$ be the bound quivers  
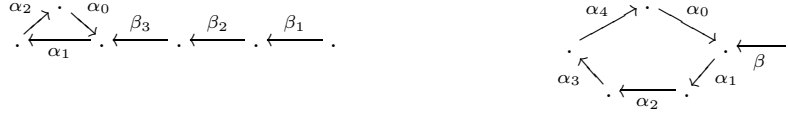
\begin{center}
 \begin{tabular}{ccccc}
$\SelectTips{eu}{10}\xymatrix@R=.6pc@C=.6pc{
		& .\ar[dr]^{\a_0}	&	&	&	&	&	&	&\\ 
.\ar[ur]^{\a_2}	&		&.\ar[ll]^{\a_1}	&	&.\ar[ll]_{\b_3}	&	&.\ar[ll]_{\b_2}	&	&.\ar[ll]_{\b_1}	}$ & && &$\SelectTips{eu}{10}\xymatrix@R=.8pc@C=.5pc{
		&		&.\ar[drr]^{\a_0}	&	&	& &\\ 
.\ar[urr]^{\a_4}	&		&		&	&.\ar[dl]^{\a_1}	&& \ar[ll]^\b\\
		&.\ar[ul]^{\a_3}	&		&.\ar[ll]^{\a_2}	&	& &}$ \\
 &  & &&\\

$I=\langle \a_i\a_{i+1},  \b_j\b_{j+1}| 0\leqslant i \leqslant 2, 1\leqslant j \leqslant 2 \rangle$&&&&$I'=\langle \a_i\a_{i+1}| 0\leqslant i \leqslant 4 \rangle$ and \\
and indices are read modulo 3.&&&& indices are read modulo 5.
   \end{tabular}
\end{center}
The algebra $A= \K Q/I$ is derived equivalent to a cluster tilted algebra of type $\mathbb{A}_6$, though it is not itself a cluster tilted algebra. The algebra $A'=\K Q/I'$ is a $3$-cluster tilted algebra of type $\mathbb{A}_6$. In both cases the determinant of the Cartan matrix is $2$, and, furthermore, the Cartan matrices are unimodular equivalent. In contrast, in case the characteristic of  $\K$ is not 2, then $\HH{6}{A} \simeq \K$, whereas $\HH{6}{A'}=0$; and $\HH{10}{A}=0$ whereas $\HH{10}{A'}\simeq \K $.
\end{ex}

\section*{Acknowledgements}
The authors gratefully thank Ibrahim Assem for several interesting and helpful discussions. This paper is part of the second author's Ph.D. thesis, done under the direction of Ibrahim Assem, she gratefully acknowledges financial support from the faculty of Sciences of the Universit\'e de Sherbrooke, and the  \emph{Institut des Sciences Math\' ematiques, I.S.M.}
 
\bibliographystyle{acm}
\bibliography{../../ReferenciaMat/biblio}
\end{document}